\documentclass[11pt]{amsart}  
   
\newtheorem{theorem}{Theorem}[section]
\newtheorem{proposition}[theorem]{Proposition}

\newtheorem{corollary}[theorem]{Corollary}
\theoremstyle{definition}
\newtheorem{definition}[theorem]{Definition}
\theoremstyle{remark}
\newtheorem{remark}[theorem]{Remark} 
 
\usepackage{amsthm}
\usepackage[utf8]{inputenc}
\usepackage[english]{babel}
\usepackage[T1]{fontenc}
\usepackage{textpos}
\usepackage{dsfont}
\usepackage{amsmath,amssymb,enumerate,mathptmx,MnSymbol}   
\usepackage{mathtools}
\usepackage[all]{xy}
\usepackage{multirow}
\usepackage{graphicx}	
\usepackage{float}
\restylefloat{table} 
\usepackage{etoolbox}

\newcommand{\inv}{^{\raisebox{.2ex}{$\scriptscriptstyle-1$}}}   

\begin{document} 
\author{Amartya Goswami}

\address{Department of Mathematics and Applied Mathematics\\
University of Johannesburg, South Africa}
\email{agoswami@uj.ac.za}
\title{$L$-homologies of double complexes}

\begin{abstract}   
The notion of $L$-homologies  (of double complexes) as proposed in this paper extends the notion of classical horizontal and vertical homologies, along with two other new homologies introduced in the homological  diagram lemma  called salamander lemma.  We enumerate all $L$-homologies associated with an object of a double complex and   provide new examples of exact sequences. We describe a classification problem of these exact sequences.   We study two poset structures on these $L$-homologies; one of them determines the trivialities of horizontal and vertical homologies of an object in terms of other $L$-homologies of that object, whereas the second structure shows the significance of the two homologies introduced in salamander lemma. Finally, we prove the existence of a faithful amnestic Grothendieck fibration from the category of $L$-homologies to a category consisting of objects and morphisms of a given double complex.  \end{abstract}

\makeatletter
\@namedef{subjclassname@2020}{%
\textup{2020} Mathematics Subject Classification}
\makeatother

\subjclass[2020]{Primary 55U05 secondary 18G35; 06A12}
\keywords{Double complexes; exact sequences;  posets; Grothendieck fibrations} 

  
\maketitle   
 
 \section{Introduction}            
In studying double complexes, we come across with horizontal and vertical homologies associated with objects of double complexes. It is in \cite{B12}, where two new homologies have been introduced to formulate a new homological diagram lemma for abelian categories. We observe that all these four types of homologies are of the form $U/V,$ where  $V\leqslant U$ follows from the properties of double complexes, and $U,$ $V$ are joins/meets of kernels and images respectively.  Given an object $A$ of a double complex, we notice that the images of all morphisms with codomain $A$ and kernels of all  morphisms with domain $A$ form two lattices. Based on these observations, we define an \emph{$L$-complex} and hence an \emph{$L$-homology} $U/V$ (see Definition \ref{def1}) having $U$ and $V$ as elements of the above mentioned lattices.  We obtain a complete list of such $L$-homologies associated with an object of a double complex (see Proposition \ref{ho2}). 
     
These $L$-homologies are accompanied with two partial order relations. One of these, (see Definition \ref{binary1}) on the set of $L$-homologies of an object,  forms a join-semilattice with a top element (see Theorem \ref{slt}) and trivialities of other $L$-homologies imply trivial horizontal and vertical homologies of that object (see Corollary \ref{pocor}). The second relation is based on canonical morphisms (in the sense of Proposition \ref{pro1}) between $L$-homologies, and the set of $L$-homologies of an object under this relation forms a bounded poset.  Using a sufficient condition of exactness of a sequence of subquotients,  we give several examples of exact sequences of $L$-homologies formulating in terms of this second partial order relation (see Proposition \ref{eexact}) and discuss about the classification problem of those exact sequences.

Finally, for a given double complex, we construct a category of  $L$-homologies of objects of the double complex and define a functor from it to a category whose objects and morphisms are the same as that of the given double complex along with the identity morphisms. This functor is a Grothendieck fibration in the sense of \cite{G59}. It is also   faithful, amnestic and has a left adjoint and a right adjoint. 
 
The key idea of this paper is  the notion of $L$-homologies, which generalizes the existing notion of homologies of double complexes and gives a way to construct exact sequences. The structures on these $L$-homologies help to understand inter-relation between them.   
\section{Preliminaries}\label{two}    
In this section, we recall some of the definitions and properties related to subobjects and subquotients in abelian categories. 
For a fuller treatment  we refer the reader to \cite{B94}, \cite{F64}, \cite{M65}, \cite{P70} and \cite{M98}.  

Let $\mathds{A}$ be an well-powered abelian category. The image, written as  $\mathsf{Im}f,$ of a morphism $f\colon A\to B$ in $\mathds{A}$ is the smallest subobject  of $B$ such that $f$ factors through the representing monomorphism $m_{\mathsf{Im}f}\colon \mathsf{Im}f\to B.$  Every morphism $f\colon A \to B$ in $\mathds{A}$ has an epi-mono factorization: $A\xrightarrow{e}\mathsf{Im}f\xrightarrow{m}B.$ A sequence of composible morphisms $A\xrightarrow{f}B\xrightarrow{g} C$ is called exact at $B$ if $\mathsf{Im}f=\mathsf{Ker}g$.
 
The set of subobjects of an object $A$ in $\mathds{A}$  forms a bounded modular lattice $(\mathsf{Sub}({A}), \leqslant , \vee, \wedge, \bot, \top)$, where for two subobjects $S$ and $S'$ of $A,$ $S\leqslant S'$ if there exists a monomorphism $\phi\colon S\to S'$ such that the monomorphism $m_S\colon S\to A$ factors through the monomorphism $m_{S'}\colon S'\to A$ by $\phi,$ i.e. $m_S=m_{S'}\phi,$
the union $S\vee S'$ of  $S$ and $S'$ is the image of the morphism $S\sqcup S'\to A,$
the intersection $S\wedge S'$ of  $S$ and $S'$ is the pullback of the morphisms $m_S\colon S\to A$ and $m_{S'}\colon S'\to A,$
the smallest subobject (or the bottom element)  $\bot$ of $\mathsf{Sub}(A)$ is the monomorphism $0\to A,$
the biggest subobject (or the top element) $\top$ of $\mathsf{Sub}(A)$ is the monomorphism $A\to A,$ and the modularity law means that for any three subobjects $X,$ $Y,$ $Z$  of $A,$  and $X\!\leqslant\! Z$ implies $X\vee (Y\wedge Z)= (X\vee Y)\wedge Z.$ 

Every morphism $f\colon A \to B$ in $\mathds{A}$ induces an adjunction between $\mathsf{Sub}(A)$ and $\mathsf{Sub}(B)$ as follows: 
$$
\xymatrix@C=3pc{
\mathsf{Sub}(A)\ar@/^/[r]^{f} \ar@{}@<1pt>[r]|{\bot} &	\mathsf{Sub}(B),\ar@/^/[l]^{f\inv}
}
$$
where, the direct image functor $f$ is defined on a subobject $S$ of $A$ as the image of the composite $S\to A\to  B,$ whereas for a subobject $T$ of $B$ the inverse image functor is defined as the pullback of the morphisms $m_T\colon T\to B$ and $f\colon A\to B.$ From the functoriality, we have $f(S)\leqslant f(S')$ and $f\inv (T)\leqslant f\inv (T')$ whenever $S\leqslant S'$ and $T\leqslant T'$ for subobejcts $S,$ $S'$ of $A$ and subobjects $T,$ $T'$ of $B.$    For a  subobject $S$ of $A$ and a subobject $T$ of $B,$ by adjunction we obtain $f(S)\leqslant T$ if and only if $S\leqslant f\inv(T).$ For subobjects $S,$ $S'$ of $A$ and subobjects $T,$ $T'$ of $B,$ we have  $f(S\vee S')=f(S) \vee f(S')$ and $f\inv(T\wedge T')=f\inv(T)\wedge f\inv(T').$ The direct and inverse image functors are related to each other as  $f\inv f(S)=S\vee \mathsf{Ker}f$ and $ff\inv(T)=T\wedge \mathsf{Im}f,$ where $S$ is a subobject of $A$ and $T$ is a subobject of $B.$

The following proposition is about existence of canonical morphisms and about holding certain commutative diagrams between subquotients. The hypothesis and the conclusions of this proposition is going to play a major role for rest of our work.  
\begin{proposition}\label{pro1}
In the diagram below: 
\begin{equation}  
\vcenter{
 \tag{1}  
\label{eq:1} 
\xymatrix@C=.5pc@R=.3pc
{
& W\!/\!R & & & X\!/\!S\ar@{..>}[lll]_(.4){h''}
\\
& & & &
\\
V\!/\!Q\ar@{..>}[uur]^(.5){f''}
& & & U\!/\!P\ar@{..>}[lll]_(.3){e''} \ar@{..>}[uur]_(.5){\!\!g''}
\\
& W\ar@{>->}'[]+<0ex,-2.5ex>;[dd]^(.4){\!\iota_W}[ddd]\ar@{->>}'[u][uuu]_(.4){\!\pi_R} & & & X\ar@{>->}[]+<0ex,-2.5ex>;[ddd]^(.3){\iota_X}\ar@{->>}[uuu]_(.4){\!\pi_S}\ar@{..>}'[l][lll]_(.4){h'}
\\
& & & &
\\
V\ar@{>->}[]+<0ex,-2.5ex>;[ddd]_(.4){\iota_V}\ar@{->>}[uuu]^{\pi_Q}\ar@{..>}[uur]^{f'}
& & & U\ar@{>->}[]+<0ex,-2.5ex>;[ddd]^(.4){\!\iota_U}\ar@{->>}[uuu]_(.4){\!\pi_P}\ar@{..>}[lll]_(.3){e'} \ar@{..>}[uur]_{\!\!\!g'}
\\
& C   & & & D \ar'[l][lll]_(.3){h}
\\
& & & &
\\
B \ar[uur]^{f} & &  & A \ar[uur]_{\!\!g}\ar[lll]^{e}
}} 
\end{equation} 
$\mathrm{(a)}$ let $e\colon A\to B$ be a morphism in $\mathds{A},$ and let $U, P$ be subobjects of $A$ with $P\leqslant U$.  Let $V, Q$  be subobjects of $B$ with $Q\leqslant V$. If $e(P)\leqslant Q$ and $U\leqslant e\inv(V)$ then there are  canonical morphisms $e'\colon U \to V$ and $e''\colon U/P\to V/Q$ such that
$
\iota_Ve'=e\iota_U$ and $e''\pi_P=\pi_Qe'.
$
   
$\mathrm{(b)}$ Let  $P\leqslant U$ be subobjects of $A,$  $Q\leqslant V$ be subobjects of $B,$  $R\leqslant W$ be subobjects of $C,$  $S\leqslant X$ be subobjects of $X.$   Let 
$e(P)\leqslant\! Q,$ $U\leqslant e\inv(V);$ $e(Q)\leqslant R,$ $V\leqslant f\inv(W);$ $g(P)\leqslant S,$ $U\leqslant g\inv(X);$ $h(S)\leqslant R,$ $X\leqslant h\inv(W),$ and $fe=hg.$ Then there are canonical morphisms $e', f', g', h'$ and $e'', f'', g'', h''$ such that
$f'e'=h'g'$ and $f''e''=h''g''.$

$\mathrm{(c)}$ Let $A\xrightarrow {e} B\xrightarrow{f} C$ be composible morphisms in $\mathds{A}.$ If
$U,$ $P$ are subobjects of $A$ with $P\leqslant U;$ $V,$ $Q$ are subobjects of $B$ with $Q\leqslant V;$ $W,$ $R$ are subobjects of $C$ with $R\leqslant W;$ and  
$e(P)\leqslant Q,$ $U\leqslant e\inv(V);$ $f(Q)\leqslant R,$ $V\leqslant f\inv(W);$ $g(P)\leqslant S,$ $U\leqslant g\inv(X),$ then the sequence $U/P\to V/Q\to W/R$ is exact if 
\begin{equation}
e(U)  \vee  Q = f\inv(R)\wedge  V.
\label{ext}
\end{equation}
\end{proposition}
\begin{proof}
(a) Since $U\leqslant e\inv(V)$ implies $e(U)\leqslant V,$ by the universal property of $\iota_V,$ there exists a unique morphism $e'\colon U\to V$ such that $\iota_Ve'=e\iota_U,$ and	since $e(P)\leqslant Q,$ by the universal property of $\pi_Q,$ there exists a unique morphism $e''\colon U/P\to V/Q$ such that $e''\pi_P=\pi_Qe'.$

(b) Applying the argument of (a) on the hypothesis, we obtain the canonical morphisms $e',$ $f',$ $g',$ $h',$  $e'',$ $f'',$ $g'',$ $h'',$ and the identities:
$$
\iota_Ve'=e\iota_U,\, e''\pi_P=\pi_Qe';\; \iota_Wf'=f\iota_V,\, f''\pi_Q=\pi_Rf';
$$
$$
\iota_Wh'=h\iota_X,\, h''\pi_S=\pi_Rh';\;
\iota_Xg'=g\iota_U,\,  g''\pi_P=\pi_Sg'.
$$ 
To obtain the identity $f'e'=h'g',$ we notice that
$$\iota_Wf'e'=f\iota_Ve'=fe\iota_U=hg\iota_U=h\iota_Xg'=\iota_Wh'g',$$
and since $\iota_W$ is a monomorphism, we have the desired identity. Finally, to obtain  $f''e''=h''g'',$ we observe that
$$f''e''\pi_P =f''\pi_Qe'= \pi_Rf'e'=\pi_Rh'g'=h''\pi_Sg'=h''g''\pi_P,$$
and since $\pi_P$ is an epimorphism, we have the desired commutativity. 

(c) The existence of the canonical morphisms $U/P\to V/Q$ and $V/Q\to W/R$  follow from (b). To have the condition (\ref{ext}),  we observe that the image of the morphism $U/P\to V/Q$ is $(e(U)\vee Q)/Q$ while the kernel of the morphism $V/Q\to W/R$ is $(f\inv(R)\wedge V)/Q.$ 
\end{proof}    
\section {$L$-homologies}    
Let us consider a portion of a double complex with objects and  morphisms as shown in the diagram below:
\begin{equation}\label{dc}
\vcenter{
\xymatrix@=1.5pc{ & &\ar@{.}[d]  &\ar@{.}[d]\\&\bullet\ar@{.}[l]\ar[r]^{x}\ar[dr]^{w}& \bullet\ar[d]^{m} & \ar[d]\\
\ar@{.}[r] & G\ar[r]^{a}\ar[dr]^{p} & C \ar[dr]^{r} \ar[r]\ar[d]^{c} & \bullet\ar[d]^{v}\ar[r]\ar[dr]^{t} &\bullet\ar[d] \ar@{.}[r] &\\
\ar@{.}[r] &	F \ar[r]_{d} & A \ar[r]^{e}\ar[d]_{f} \ar[dr]^{\!q} & B\ar[r]^{b}\ar[dr]^{s} \ar[d]^{g}& E\ar[r]^{h}\ar[d]\ar[dr]^{u} &\bullet\ar@{.}[r] &\\
\ar@{.}[r] & \ar[r]	&I\ar[r]\ar@{.}[d] & D\ar[r]\ar@{.}[d] & \bullet\ar[r]\ar@{.}[d] & \bullet \ar@{.}[r] &\\
&&&& 
}} 
\end{equation}
where $w\!=\!mx,$ $p\!=\!ca,$ $r\!=\!ec,$  $q\!=\!ge,$ $t\!=\!bv.$ For every object $A$ of (\ref{dc}), consider the diagram of canonical monomorphisms induced by the double complex structure: 
\begin{equation}\label{a} 
\begin{gathered} 
\newdir{(>}{{}*!/-5pt/\dir{>}}
\xymatrix@C=.2pc@R=.3pc
{
&\mathsf{Ker}e\vee \mathsf{Ker}f\ar@{(>->}[dr] & & & \mathsf{Ker}f\;\;\ar@{(>->}[dll]\ar@{(>->}[lll]
\\
& &\mathsf{Ker}q & &
\\
\mathsf{Ker}e\ar@{(>->}[urr]\ar@{(>->}[uur]
& & &\mathsf{Ker}e\wedge \mathsf{Ker}f\ar@{(>->}[ul]\ar@{(>->}[lll] \ar@{(>->}[uur]
\\
& \mathsf{Im}c\vee \mathsf{Im}d \ar@{(>->}'[u][uuu]  & & &\mathsf{Im}c\ar@{(>->}[uuu] \ar@{(>->}'[l][lll] 
\\
& &\mathsf{Im}p\;\ar@{(>->}[ul]\;\ar@{(>->}[urr]\ar@{(>->}[dll]\ar@{>->}[dr]& &
\\
\mathsf{Im}d \ar@{(>->}[uur]\ar@{(>->}[uuu] & &  &\mathsf{Im}c\wedge  \mathsf{Im}d \ar@{(>->}[uur]\ar@{(>->}[lll]\ar@{(>->}[uuu]
}
\end{gathered}  
\end{equation} 
\textit{Notation}. We denote the  lattice generated by the subobjects $\mathsf{Ker}e,$  $\mathsf{Ker}f,$ and $\mathsf{Ker}q$ by $\mathsf{L}_{\mathsf{Ker}};$ while   $\mathsf{L}_{\mathsf{Im}}$ denotes the lattice generated by the subobjects $\mathsf{Im}c,$  $\mathsf{Im}d,$ and $\mathsf{Im}p.$ 
\begin{definition}  
An \emph{$L$-complex} associated with an object $A$ of a double complex $(\ref{dc})$ is an ordered pair of sets $(M_c,M_d)$ of morphisms such that (i) $A$ is the codomain for each $f\in M_c,$ (ii) $A$ is the domain for each $g\in M_d,$ and for all $f\in M_c,$ $g\in M_d,$ $\mathsf{Im}f \leqslant \mathsf{Ker}g.$  \end{definition} 
\begin{definition}\label{def1}  
An \emph{$L$-homology}   of an object $A$ of a double complex $(\ref{dc})$ is  a triple $(U, V, U\to U/V)$ such that
$U\in\mathsf{L}_{\mathsf{Ker}},$ 
$V\in\mathsf{L}_{\mathsf{Im}},$ and 
$V\leqslant U.$  We say an $L$-homology $U/V$ is \emph{trivial} if $U/V$ is a zero object.  
\end{definition} 
\textit{Notation}. For simplicity, we use the notation $U/V$ to denote an $L$-homology  instead of the triple as mentioned in Definition \ref{def1}.  We denote an $L$-homology and the set of $L$-homologies of an object $A$ of a double complex $\mathfrak{D}$ by $H_A$ and $\mathsf{LH}_A$  respectively. The set of all $L$-homologies of  $\mathfrak{D}$ will be denoted by $\mathsf{LH}_{\mathfrak{D}}.$ 

\begin{proposition}\label{ho2}	
For every object $A$ of a double complex $(\ref{dc}),$ there are eighteen $L$-homologies associated with fourteen $L$-complexes as described in the Table\,\ref{tabA} below: 
\end{proposition} 
\begin{table} 
\begin{center}
{
\setlength\tabcolsep{1pt}  
\begin{tabular}{|c|c|c|c|}
\hline  
$L\text{-} \mathrm{homology}$ &  $L\text{-} \mathrm{complex}$&$L\text{-} \mathrm{homology}$ &  $L\text{-} \mathrm{complex}$\\     
\hline 
\multirow{2}{9em}{$A_{\mathsf{h}}\!=\!\mathsf{Ker} e/\mathsf{Im} d$} & $\xymatrix@R=1pc@C=1pc{\bullet\ar[r]^{d}&A\ar[r]^{e} & \bullet }$ &\multirow{2}{11em}{$^\vee \!\!A_{\mathsf{d}}\!=\!(\mathsf{Ker} e \vee  \mathsf{Ker}f)/\mathsf{Im}p$}  &$\xymatrix@R=1pc@C=1pc{\bullet\ar[dr]^{p}\\&A\ar[d]^{f}\ar[r]^{e}&\bullet\\&\bullet}$ \\
\hline
\multirow{2}{11em}{$^\vee \!\!A_{\mathsf{h}}\!=\!(\mathsf{Ker} e \vee  \mathsf{Ker}f)/\mathsf{Im}d$}  & $\xymatrix@R=1pc@C=1pc{\bullet\ar[r]^{d}&A\ar[d]^{f}\ar[r]^{e} &\bullet\\&\bullet }$&\multirow{2}{7.6em}{$A_{\mathsf{d}}\!=\!\mathsf{Ker} q/\mathsf{Im}p$}   & $\xymatrix@R=.5pc@C=.5pc{\bullet\ar[dr]^{p}\\&A\ar[dr]^{q}&\\&&\bullet}$   \\   
\hline
\multirow{2}{8em}{$^{\mathsf{d}}\!\!A_{\mathsf{h}}\!=\!\mathsf{Ker} q / \mathsf{Im}d$}  &  $\xymatrix@R=1pc@C=1pc{\bullet\ar[r]^{d}&A\ar[dr]^{q} &\\&&\bullet}$&\multirow{2}{13em}{$ ^\vee \!\!A_{\star}\!=\!(\mathsf{Ker}e \vee  \mathsf{Ker}f)\!/\!(\mathsf{Im} c \vee  \mathsf{Im}d)$} &$\xymatrix@R=1pc@C=1pc{&\bullet\ar[d]^{c}\\\bullet\ar[r]^{d}&A\ar[d]^{f}\ar[r]^{e}&\bullet\\&\bullet}$\\
\hline
\multirow{2}{8em}{$A_{\mathsf{v}}\!=\!\mathsf{Ker} f/\mathsf{Im} c$}  & $\xymatrix@R=1pc@C=1pc{\bullet\ar[d]^{c}\\A\ar[d]^{f}\\\bullet  }$&\multirow{2}{10.43em}{$A_{\star} \!=\!\mathsf{Ker} q/(\mathsf{Im} c \vee  \mathsf{Im}d)$}&$\xymatrix@R=1pc@C=1pc{&\bullet\ar[d]^{c}\\\bullet\ar[r]^{d}&A\ar[dr]^{q}&\\&&\bullet}$   \\
\hline
\multirow{2}{11em}{$^\vee \!\!A_{\mathsf{v}}\!=\!(\mathsf{Ker} e \vee  \mathsf{Ker}f)/\mathsf{Im}c$}  & $\xymatrix@R=1pc@C=1pc{\bullet\ar[d]^{c}\\A\ar[d]^{f}\ar[r]^{e}&\bullet\\\bullet}$&\multirow{2}{12em}{$ ^\star\!\!A_{\wedge }\!=\!(\mathsf{Ker} e \wedge  \mathsf{Ker}f)/(\mathsf{Im} c \wedge  \mathsf{Im}d)$} &$\xymatrix@R=1pc@C=1pc{&\bullet\ar[d]^{c}\\\bullet\ar[r]^{d}&A\ar[d]^{f}\ar[r]^{e}&\bullet\\&\bullet}$  \\
\hline
\multirow{2}{8em}{$^{\mathsf{d}}\!\!A_{\mathsf{v}}\!=\!\mathsf{Ker} q/ \mathsf{Im}c$}   &$\xymatrix@R=1pc@C=1pc{\bullet\ar[d]^{c}\\A\ar[dr]^{q}&\\&\bullet}$&\multirow{2}{10.43em}{$^{\mathsf{h}}\!\!A_{\wedge }\!=\!\mathsf{Ker} e/(\mathsf{Im} c \wedge  {\mathsf{Im}}d)$}&$\xymatrix@R=1pc@C=1pc{&\bullet\ar[d]^{c}\\\bullet\ar[r]^{d}&A\ar[r]^{e}&\bullet}$    \\
\hline
\multirow{2}{11em}{$^\star\!\!A\!=\!(\mathsf{Ker} e \wedge  \mathsf{Ker}f)/\mathsf{Im} p$}  &$\xymatrix@R=1pc@C=1pc{\bullet\ar[dr]^{p}\\&A\ar[d]^{f}\ar[r]^{e}&\bullet\\&\bullet}$&\multirow{2}{11em}{$^{\mathsf{v}}\!\!A_{\wedge }\!=\!\mathsf{Ker} f/(\mathsf{Im} c \wedge  {\mathsf{Im}}d)$}&$\xymatrix@R=1pc@C=1pc{&\bullet\ar[d]^{c}\\\bullet\ar[r]^{d}&A\ar[d]^{f}\\&\bullet}$    \\
\hline
\multirow{2}{8em}{$^{\mathsf{h}}\!\!A_{\mathsf{d}}\!=\! \mathsf{Ker} e/\mathsf{Im} p$} & $\xymatrix@R=1pc@C=1pc{\bullet\ar[dr]^{p}\\&A\ar[r]^{e}&\bullet}$&\multirow{2}{10.43em}{$^{\mathsf{d}}\!\!A_{\wedge }\!=\!\mathsf{Ker} q /(\mathsf{Im} c \wedge  {\mathsf{Im}}d)$}&$\xymatrix@R=1pc@C=1pc{&\bullet\ar[d]^{c}\\\bullet\ar[r]^{d}&A\ar[dr]^{q}&\\&&\bullet}$  \\
\hline
\multirow{2}{9em}{$^{\mathsf{v}}\!\!A_{\mathsf{d}}\!=\!\mathsf{Ker} f/\mathsf{Im} p$} & $\xymatrix@R=1pc@C=.6pc{\bullet\ar[dr]^{p}\\&A\ar[d]^{f}&\\&\bullet}$&\multirow{2}{12em}{$^{\vee }\!\!A_{\wedge }\!=\!(\mathsf{Ker} e\vee \mathsf{Ker}f)/(\mathsf{Im} c \wedge  {\mathsf{Im}}d)$}&$\xymatrix@R=1pc@C=1pc{&\bullet\ar[d]^{c}\\\bullet\ar[r]^{d}&A\ar[d]^{f}\ar[r]^{e}&\bullet\\&\bullet}$  \\
\hline  
\end{tabular}} 
\end{center}
\caption{$L$-complexes and $L$-homologies}
\label{tabA}
\end{table} 

\begin{proof}
From the diagram (\ref{a}) we notice that in order to satisfy condition $V\leqslant U$ of the Definition \ref{def1}, there are eighteen possibilities of $U/V$ as described in the Table \ref{tabB} below:
\begin{table}[H]  
\begin{center}  
\begin{tabular}{|c|c|c|}
\hline 
$V$&$U$& $\#$
\\ 
\hline
$\mathsf{Im}d$ & $\mathsf{Ker}e,  \mathsf{Ker}q, \mathsf{Ker}e\vee  \mathsf{Ker}f$ & $3$ \\
\hline 
$\mathsf{Im}c$ & $\mathsf{Ker}f,  \mathsf{Ker}q, \mathsf{Ker}e\vee  \mathsf{Ker}f$ & $3$
\\
\hline 
$\mathsf{Im}p$ & $\mathsf{Ker}e,  \mathsf{Ker}f,  \mathsf{Ker}q, \mathsf{Ker}e\vee  \mathsf{Ker}f, \mathsf{Ker}e\wedge  \mathsf{Ker}f$ & $5$
\\
\hline 
$\mathsf{Im}c \vee  \mathsf{Im}d$ & $  \mathsf{Ker}q, \mathsf{Ker}e\vee  \mathsf{Ker}f$ & $2$
\\
\hline 
$\mathsf{Im}c \wedge  \mathsf{Im}d$ & $\mathsf{Ker}e,  \mathsf{Ker}f,  \mathsf{Ker}q, \mathsf{Ker}e\vee  \mathsf{Ker}f, \mathsf{Ker}e\wedge  \mathsf{Ker}f$ & $5$
\\  
\hline
\end{tabular}
\end{center} 
\caption{$L$-complexes and $L$-homologies}
\label{tabB}
\end{table}
\end{proof} 
\begin{remark}
The $L$-homologies $A_{\star}$ and $^\star \!\!A$ were introduced in \cite{B12} (with notation $A_{\square}$ and $^\square \!\!A$ respectively). For the rest of the $L$-homologies, we have used the following notation convention: $A_{\mathsf{h}},$ $A_{\mathsf{v}},$ and $A_{\mathsf{d}}$ respectively denote the horizontal, vertical, and diagonal $L$-homologies. Whenever $\mathsf{h},$ $\mathsf{v},$ and $\mathsf{d}$ appear as subscripts or superscripts (along with other symbols as superscripts or subscripts respectively), then they mean the denominator or numerator expressions of the $A_{\mathsf{h}},$ $A_{\mathsf{v}},$ and $A_{\mathsf{d}}$ respectively. The same holds with respect to subscript or superscript $\star,$ whereas $\vee$ and $\wedge$ respectively mean the usual joins and meets. 
\end{remark}
\begin{proposition}
For a given double complex, the $L$-homologies $X_{\mathsf{d}}$ of objects $X$ of the double complex induce a double complex.
\end{proposition}
\begin{proof}
Consider the double complex  (\ref{dc}). First, we show that $A_{\mathsf{d}}\to B_{\mathsf{d}}\to E_{\mathsf{d}}$ is a horizontal complex. The proof of the vertical complex is similar. The canonical morphism $A_{\mathsf{d}}\to B_{\mathsf{d}}$ (and similarly $B_{\mathsf{d}}\to E_{\mathsf{d}}$) exists because $e(\mathsf{Ker}q)= (\mathsf{Im}e\wedge  \mathsf{Ker}g)\leqslant \mathsf{Ker}s$ and $e(\mathsf{Im}p)\leqslant e(\mathsf{Im}c)=\mathsf{Im}r.$    To have the complex, we notice that
$$
e(\mathsf{Ker}q)\vee \mathsf{Im}r=(\mathsf{Im}e\wedge  \mathsf{Ker}g)\vee \mathsf{Im}r
= \mathsf{Im}e\wedge  \mathsf{Ker}g
\leqslant \mathsf{Ker}b\wedge  \mathsf{Ker}s 
\leqslant (\mathsf{Im}v \vee  \mathsf{Ker}b)\wedge  \mathsf{Ker}s.
$$ 
In order to obtain a double complex, what remains is to show the commutativity of a square having vertices as $U_{\mathsf{d}},$ $V_{\mathsf{d}},$ $W_{\mathsf{d}},$ and $X_{\mathsf{d}}.$ By the above argument once we have the canonical morphisms  $U_{\mathsf{d}}\to V_{\mathsf{d}}\to W_{\mathsf{d}}$ and $U_{\mathsf{d}}\to X_{\mathsf{d}}\to W_{\mathsf{d}},$  the commutativity follows from the  Proposition \ref{pro1}(a). 
\end{proof} 
In next section, we are going to describe two poset structures of $L$-homologies and the category of $L$-homologies.
\section{Structures of $L$-homologies}
\begin{definition}\label{binary1}
Let $H_A=X/Y$ and $H'_A=U/V$ be two elements of $\mathsf{LH}_A.$  We define a  relation $\hookrightarrow$ on  $\mathsf{LH}_A$  as $H_A\hookrightarrow H_A'$ if  $X\leqslant U$ and $V\leqslant Y.$ We define the join of $H_A$ and $H'_A$ as $H_A\curlyvee H_A'=  (X\vee U)/(Y\wedge V).$   
\end{definition}       
\begin{theorem} \label{slt}   
For each object $A$ of a double complex   $\mathfrak{D},$ the pair $(\mathsf{LH}_A, \hookrightarrow)$ forms a join-semilattice with $A_{\mathsf{d}}$ as the top element.
\end{theorem}                  
\begin{proof}          
The partial order relation $\leqslant,$ and the definition of the relation $\hookrightarrow$ shows that $(\mathsf{LH}_A, \hookrightarrow)$ is a poset.  From the inclusion diagram (\ref{a}) we notice that for $X, U\in  \mathsf{L}_{\mathsf{Ker}}$ and $Y, V\in \mathsf{L}_{\mathsf{Im}}$ we have $Y\wedge V\leqslant X\vee U.$ Therefore, the join of two elements of $\mathsf{LH}_A$ exists. We observe that  ${U\leqslant \mathsf{Ker}q}$ and $ \mathsf{Im}p\leqslant V$ for any double complex homology $U/V\in  \mathsf{LH}_A,$  i.e. $A_{\mathsf{d}}$ is the top element of $(\mathsf{LH}_A, \hookrightarrow)$ as asserted.
\end{proof} 
\begin{corollary}\label{pocor}           
$\mathrm{(a)}$ If any one of the $L$-homologies $^\mathsf{h}\!\!A_{\wedge },$ $^\mathsf{h}\!\!A_{\mathsf{d}},$ $^\vee \!\!A_{\mathsf{d}},$ $^\vee \!\!A_{\mathsf{h}},$
$^{\mathsf{d}}\!\!A_{\mathsf{h}},$
$^{\mathsf{d}}\!\!A_{\wedge },$  $A_{\mathsf{d}}$ is trivial then so is $A_{\mathsf{h}}.$

$\mathrm{(b)}$ If any one of the $L$-homologies $^\mathsf{v}\!\!A_{\wedge },$ $^\mathsf{v}\!\!A_{\mathsf{d}},$ $^\vee \!\!A_{\mathsf{d}},$ $^\vee \!\!A_{\mathsf{v}},$ 
$^{\mathsf{d}}\!\!A_{\mathsf{v}},$
$^{\mathsf{d}}\!\!A_{\wedge },$  $A_{\mathsf{d}}$  is trivial then so is $A_{\mathsf{v}}.$

$\mathrm{(c)}$ The triviality of  $A_{\mathsf{d}}$  implies trivialities of all $L$-homologies of $A$.
\end{corollary}
\begin{remark}  
The meet $X/Y\curlywedge U/V= (X\wedge U)/(Y\vee V)$ of two elements $X/Y$ and $U/V$ of   $\mathsf{LH}_\mathfrak{D}$ may not exist. For example, the meet of $A_{\mathsf{h}}$ and $A_{\mathsf{v}}$ does not exist because $ \mathsf{Im}c\vee \mathsf{Im}d \nleq   \mathsf{Ker}e\wedge \mathsf{Ker}f .$
\end{remark}  
\begin{definition}\label{binary2}
Let $f\colon A\to B$ be a morphism of $\mathfrak{D}$ and $H_A,$  $H_B$ be elements of $\mathsf{LH}_A$ and $\mathsf{LH}_B$ respectively.  We define a relation $\preceq$  on  $\mathsf{LH}_{\mathfrak{D}}$ as    
$H_A\preceq H_B$ if there exists a canonical  morphism $\phi\colon H_A \to H_B$ in the sense of Proposition \ref{pro1}.     
\end{definition}
Given an object $A$ of $\mathfrak{D},$ we observe that the relation $\hookrightarrow $ in Definition \ref{binary1} and the relation $\preceq$   in Definition \ref{binary2} are inter-related as follows. 
\begin{proposition} Given $L$-homologies $H_A=X/Y$ and $H_A'=U/V$ of an  object $A$  
we have $H_A\hookrightarrow H_A'\Leftrightarrow H_A\preceq H_A'$  if and only if $Y=V.$     
\end{proposition}   
\begin{theorem}\label{poset}
For each object $A$ of the double complex (\ref{dc}),  the pair $(\mathsf{LH}_A, \preceq)$  forms a poset with  bottom and top elements respectively as $^\star\!\!A$ and $A_{\star}.$
\end{theorem}  
\begin{proof}
The fact that the relation $\preceq$  is partially ordered, follows from the Proposition \ref{pro1}(a) and Definition \ref{binary2}. To find the top and bottom elements, let $U/V\in\mathsf{LH}_A.$ From the inclusion diagram (\ref{a}), we obtain the following inclusions: 
$$
\mathsf{Ker}e \wedge  \mathsf{Ker}f\leqslant U\leqslant \mathsf{Ker}q,\quad 
\mathsf{Im}p\leqslant V\leqslant \mathsf{Im}c\vee \mathsf{Im}d,
$$
and by applying Proposition \ref{pro1}(a), we have $$(\mathsf{Ker}e \wedge  \mathsf{Ker}f)/\mathsf{Im}p \preceq  U/V \preceq  \mathsf{Ker}q/(\mathsf{Im}c\vee \mathsf{Im}d).$$ In other words,  $^\star\!\!A$ and $A_{\star}$  are  respectively the bottom and the top elements of the poset  $(\mathsf{LH}_A, \preceq)$  as asserted.
\end{proof}   
\begin{proposition}
Given a morphism $e\colon A\to B$ of $(\ref{dc})$,   the  trivialities of  $A_{\mathsf{h}}$ and $B_{\mathsf{h}}$ imply an isomorphism between the top element $A_{\star}$ of $(\mathsf{LH}_A, \preceq)$  and the bottom element $^\star\!\!B$ of $(\mathsf{LH}_B,\preceq).$ 
\label{b12}
\end{proposition} 
\begin{proof}
The existence of the connecting morphism 
$A_{\star}\to ^{\star}\!\!\!\!B$ follows from Proposition \ref{pro1}(a). Since $B_{\mathsf{h}}=1$, we obtain
$e(\mathsf{Ker}q)=\mathsf{Im}e \wedge  \mathsf{Ker}g=\mathsf{Ker}b \wedge  \mathsf{Ker}g$  whereas because of $A_{\mathsf{h}}=1,$ we get $ e\inv(\mathsf{Im}r) = \mathsf{Im}c \vee  \mathsf{Ker}e=\mathsf{Im}c\vee  \mathsf{Im}d.$ This proves the required isomorphism.  
\end{proof} 
\begin{remark} 
In \cite{B12}, the Proposition \ref{b12} has been derived as a corollary of the salamander lemma. 
\end{remark}
\subsection{Classification problem of exact sequences of $L$-homologies}
For a given double complex $\mathfrak{D}$ and a composable pair of morphisms $X\xrightarrow{f} Y\xrightarrow {g}Z$ of $\mathfrak{D},$ we observe that there are four possible types of exact sequences involving $L$-homologies of $X,$ $Y,$ and $Z:$
$$\mathrm{(I)}\; H_X\preceq H_X \preceq H_X, \mathrm{(II)}\; H_X\preceq H_X \preceq H_Y, \mathrm{(III)}\; H_X\preceq H_Y\preceq H_Y, \mathrm{(IV)}\; H_X\preceq H_Y \preceq H_Z.$$
The Proposition \ref{eexact} below gives examples of exact sequences of $L$-homologies of the types  (II) $H_X\preceq H_X\preceq H_Y$ and (III) $H_X\preceq H_Y\preceq H_Y,$ where $X$ and $Y$ are two different objects of a double complex $\mathfrak{D}.$   
\begin{proposition} 
\label{eexact}
In the double complex $(\ref{dc})$, the following sequences are exact:
\\
$\mathrm{II.}$ 
$\mathrm{(a)}$ $A_{\mathsf{h}}\preceq A_{\star}\preceq\, ^\star\!B;$
$\mathrm{(b)}$ $^{\mathsf{h}}\!\!A_{\mathsf{d}}\preceq\, ^{\mathsf{d}}\!\!A_{\mathsf{v}}\preceq\, ^{\mathsf{v}}\!B_{\mathsf{d}};$ $\mathrm{(c)}$ $^{\mathsf{h}}\!\!A_{\mathsf{d}}\preceq A_{\star}\preceq\, ^{\vee }\!B_{\mathsf{d}};$
$\mathrm{(d)}$ $C_{\mathsf{v}}\preceq C_{\star}\preceq \, ^\star \!\!A;$ \\ $\mathrm{(e)}$ $^{\mathsf{v}}\!C_{\mathsf{d}}\preceq\, ^{\mathsf{d}}\! C_{\mathsf{h}}\preceq\, ^{\mathsf{h}}\!\!A_{\mathsf{d}};$ $\mathrm{(f)}$ $^{\mathsf{v}}\!C_{\mathsf{d}}\preceq  C_{\star} \preceq\, ^\vee \!\!\!A_{\mathsf{d}}$.
\\
$\mathrm{III.}$  
$\mathrm{(a)}$ $C_{\star} \preceq  A_{\mathsf{h}}\preceq A_{\star};$
$\mathrm{(b)}$ $C_{\star} \preceq\,  ^\star \!\!A \preceq A_{\mathsf{v}};$ 
$\mathrm{(c)}$ $C_{\star} \preceq\,   ^\mathsf{h} \!\!A_{\mathsf{d}}\preceq\, ^\mathsf{d}\!\!A_{\mathsf{v}};$
$\mathrm{(d)}$ $A_{\star} \preceq  B_{\mathsf{v}}\preceq  B_{\star};$\\
$\mathrm{(e)}$ $A_{\star} \preceq\,  ^\star \!B\preceq B_{\mathsf{h}};$
$\mathrm{(f)}$ $A_{\star} \preceq\,  ^\mathsf{v} \!B_{\mathsf{d}}\preceq\, ^\mathsf{d}\!B_{\mathsf{h}}.$
\end{proposition}
\begin{remark}
The examples II(a) and III(a) are the segments of the salamander lemma as proposed in \cite{B12}. 
\end{remark}
\begin{proof}
To prove the existence of a connecting morphism of a sequence, we apply  Proposition \ref{pro1}(a),  whereas for the exactness we verify the condition (\ref{ext}) of the Proposition \ref{pro1}(c). To show the working scheme, here we only prove the exactness of II(a) and III(a).  
	
II(a): For the left hand side of (\ref{ext}), we have 
$\mathsf{Ker}e \vee  (\mathsf{Im}c \vee  \mathsf{Im}d) = \mathsf{Ker}e \vee  \mathsf{Im}c,$
whereas for the right hand side of (\ref{ext}), we get
$$e\inv(\mathsf{Im}r) \wedge  \mathsf{Ker}q=e\inv(e(\mathsf{Im}c))\wedge  \mathsf{Ker}q= (\mathsf{Ker}e \vee  \mathsf{Im}c)\wedge  \mathsf{Ker}q =\mathsf{Ker}e \vee  \mathsf{Im}c.$$

III(a): For the left hand side of  (\ref{ext}), we have $c(\mathsf{Ker}r)  \vee  \mathsf{Im}d =(\mathsf{Ker}e \wedge  \mathsf{Im}c)\vee  \mathsf{Im}d,$ and for the right hand side of (\ref{ext}), we get  $(\mathsf{Im}c \vee  \mathsf{Im}d)\wedge  \mathsf{Ker}e= (\mathsf{Ker}e \wedge  \mathsf{Im}c)\vee  \mathsf{Im}d$ (by modularity).
\end{proof} 
\begin{remark} 
The inclusion diagram (\ref{a}) and the exactness condition (\ref{ext}) confirm the non-existence of any exact sequence of the type (I); whereas, for three different objects $X,$ $Y,$ and $Z$ of a double complex $\mathfrak{D},$  the existence of exact sequence of the type (IV) is unknown.      
\end{remark}           
\begin{definition}
Given a double complex $\mathfrak{D},$ we define

(i) a category  $\mathds{C}\mathrm{om}_{\mathfrak{D}}$ consisting  objects  of $\mathfrak{D}$ and  morphisms of $\mathfrak{D}$ along with identity morphisms of each object of $\mathfrak{D};$  

(ii) a category    $\mathds{H}\mathrm{lg}_{\mathfrak{D}}$ having  objects as $L$-homologies of  $\mathfrak{D}$ and morphisms as canonical morphisms between $L$-homologies (in the sense of Proposition \ref{pro1}(a)).  
\end{definition}
\begin{definition}\label{func} 
We define a functor $\mathcal{F}\colon \mathds{H}\mathrm{lg}_{\mathfrak{D}}\to \mathds{C}\mathrm{om}_{\mathfrak{D}}$
as follows: 
$$
\mathcal{F}(H_A)=A; \quad
\mathcal{F}(H_A\to H_B)=\begin{cases}
1_A, &  A= B,\\
A\to B,& \mathrm{otherwise}.  
\end{cases}
$$
For every object $A$ of $\mathds{C}\mathrm{om}_{\mathfrak{D}},$ the \emph{fibre} of $\mathcal{F}$ at $A,$ denoted by $\mathsf{Fib}(A),$ is the subcategory of $\mathds{H}\mathrm{lg}_{\mathfrak{D}}$ consisting of those
objects and morphisms in $\mathds{H}\mathrm{lg}_{\mathfrak{D}}$ which, by $\mathcal{F},$
are mapped to $A$ and $1_A,$ respectively.
\end{definition}
\begin{remark}
We notice that $\mathsf{Fib}(A)$ 
is nothing but the set $\mathsf{LH}_A$ of $L$-homologies of the object $A.$ 
\end{remark}
\begin{definition}\label{gf}
A functor $\mathcal{G} \colon \mathds{A} \to \mathds{X}$ is said to be a \emph{Grothendieck fibration}, if for every object $S$ in $\mathds{A},$ and every morphism $\phi \colon X \to \mathcal{G}(S)$ in $ \mathds{X},$ there exists a morphism $f \colon A \to S$ in $\mathds{A}$ such that:

(a) ${\mathcal{G}(f) = \phi}$ (and in particular $\mathcal{G}(A) = X$);

(b) given a morphism $g \colon B \to S$ in $\mathds{A}$ and a morphism $\psi \colon \mathcal{G}(B) \to X$ in $\mathds{X}$ with $\mathcal{G}(g) = \phi \psi,$ there exists a unique morphism $h \colon B \to A$ with $\mathcal{G}(h) = \psi$ and $fh = g.$ 
\end{definition}
\begin{theorem}
The functor $\mathcal{F}$ in (\ref{func}), is a faithful amnestic Grothendieck fibration. Furthermore, $\mathcal{F}$ has a left adjoint and a right adjoint:
$$
\xymatrix@=70pt{ \mathds{H}\mathrm{lg}_{\mathfrak{D}}\ar[d]|{\mathcal{F}}\\   \mathds{C}\mathrm{om}_{\mathfrak{D}}\ar@/^25pt/[u]_{\mathcal{I}}\ar@{}@<11pt>[u]|{\dashv}\ar@/_25pt/[u]^-{\mathcal{T}}\ar@{}@<-11pt>[u]|{\dashv}
 } 
$$   
where to each object $A$ of $\mathds{C}\mathrm{om}_{\mathfrak{D}},$ the functors $\mathcal{I}$ and $\mathcal{T}$ assign respectively the bottom and top elements of $\mathsf{Fib}(A).$   
\end{theorem}  
\begin{proof}
In the category $\mathds{H}\mathrm{lg}_{\mathfrak{D}},$ we notice that whenever a morphism exists between any two $L$-homologies, it is unique. Therefore, the functor $\mathcal{F}$ obviously satisfy the property: $(\mathcal{F}(f)=\mathcal{F}(g))\Rightarrow(f=g),$ for any two parallel morphisms $f$ and $g$ in $\mathds{H}\mathrm{lg}_{\mathfrak{D}}$, i.e. $\mathcal{F}$ is faithful, as asserted.

Let $f\colon U/V\to U'/V'$ be an isomorphism in $\mathds{H}\mathrm{lg}_{\mathfrak{D}}$ such that $\mathcal{F}(f)=1_A.$ By definition of $\mathcal{F},$ this implies $U/V$ and $U'/V'$ are in $\mathsf{Fib}(A),$ and isomorphism forces  $f=1_{U/V},$ i.e. $\mathcal{F}$ is amnestic. 
 
To prove $\mathcal{F}$ is a Grothendieck fibration, let us consider the following two cases to verify the conditions of the Definition \ref{gf}:

$\phi$ an identity morphism: We take $f$ as $1_S,$ and $h$ will be the unique canonical morphism $g$ itself. 
  
$\phi$ is not an identity morphism: We take $A$ as $^{\star}\!\!X$ and the morphism $f$ is the (unique) composite\\ $^{\star}\!\!X\to  X_{\star}\to  ^{\star}\!\!\!\!S\to S.$  Since for any morphism $P\to Q$ in $\mathds{C}\text{om}_{\mathfrak{D}}$ there exists a unique canonical morphism $P_{\star}\to ^{\star}\!\!\!\!Q$ in  $\mathds{H}\mathrm{lg}_{\mathfrak{D}},$ the existence of the unique morphism $h$ with the desired property follows.  
    
We prove the adjunction $\mathcal{I}\dashv \mathcal{F},$ and the proof of $\mathcal{F}\dashv \mathcal{T} $ is similar. To prove $\mathcal{I}\dashv \mathcal{F},$ it is sufficient to prove the map
$$\phi_{X,H_A}\colon \mathrm{hom}_{\mathds{H}\mathrm{lg}_{\mathfrak{D}}} (^{\star}\!X,H_A)\to \mathrm{hom}_{\mathds{C}\text{om}_{\mathfrak{D}}} (X, A)$$
is a bijection for each object $H_A$ in $\mathds{H}\mathrm{lg}_{\mathfrak{D}}.$ But the bijection follows from the fact that for each $H_A,$ there is a unique canonical morphism $f\colon ^{\star}\!\!X\to H_A$ and its image $\phi_{X,H_A}(f)$ (by definition of $\mathcal{F}$) is the map $X\to A$ of $\mathds{C}\text{om}_{\mathfrak{D}}.$
\end{proof}  
\begin{remark}
The functor $\mathcal{F}$ is an example of a  Z. Janelidze form in the sense of  \cite{JW14}. For further applications of Janelidze forms, see \cite{GJ19} and the references given there.   
\end{remark}   
  
\section*{Acknowledgement}
 I thank George Janelidze for useful comments and  suggestions.

\end{document}